\documentclass[12pt]{article}

\usepackage{latexsym,amssymb,amsmath}
\usepackage[dvips]{graphicx}
\newcommand{\C}{\mathbb C}
\newcommand{\Z}{\mathbb Z}
\newcommand{\N}{\mathbb N}
\newcommand{\R}{\mathbb R}
\newcommand{\Q}{\mathbb Q}
\newcommand{\F}{\mathbb F}

\newcommand{\sma}{\left(\begin{array}}
\newcommand{\fma}{\end{array}\right)}

\newtheorem{lem}{Lemma}[section]
\newtheorem{defn}[lem]{Definition}

\newtheorem{co}[lem]{Corollary}
\newtheorem{thm}[lem]{Theorem}
\newtheorem{prop}[lem]{Proposition}

\newenvironment{proof}{\textbf{Proof.}}{\newline\hspace*{\fill}{$\Box$}\\}

\begin{document}
\title{Groups acting purely loxodromically on products of hyperbolic graphs}
\author{J.\,O.\,Button\\
Selwyn College\\
University of Cambridge\\
Cambridge CB3 9DQ\\
U.K.\\
\texttt{j.o.button@dpmms.cam.ac.uk}}
\date{}
\maketitle
\begin{abstract}
We consider the class of countable groups possessing an action on a finite
product of hyperbolic graphs where every infinite order element acts
loxodromically. When the graphs are locally finite, we obtain strong
structure theorems for the groups in this subclass, so that mapping
class groups of genus at least 3 (and $Aut(F_n)$ and $Out(F_n)$ for $n\geq 4$)
are not in this subclass. This contrasts with the general case, where 
Bestvina, Bromberg and Fujiwara showed the existence of proper actions 
of mapping class groups on a finite product of quasitrees. In particular
these quasitrees cannot be locally finite.
\end{abstract}

\section{Introduction}

It is a strong condition for a group $G$ to act properly and cocompactly
on either a hyperbolic or a CAT(0) space. We obtain the class of hyperbolic
groups from the former and the CAT(0) groups from the latter (though we
do not know that all hyperbolic groups are CAT(0)). As these two properties
do not pass to arbitrary subgroups (or even arbitrary finitely presented
subgroups), we might be tempted to remove the cocompactness requirement
to obtain a condition that is clearly inherited by subgroups and which we hope
might provide some information about the structure of $G$ as a group.

However this does not work. It is an open question whether every countable
group acts properly on some CAT(0) space, whereas every countable group
$G$ does act properly on some hyperbolic space. This can be seen from the
hyperbolic cone construction in \cite{grman}, which produces a
locally finite hyperbolic graph on which $G$ acts freely and properly
by automorphisms, and hence by isometries. As no information about $G$ can
possibly be obtained by the existence merely of a proper action on a hyperbolic
space, we would like to impose extra constraints on this action which are
useful but which fall far short of cocompactness.

When the space $X$ is CAT(0) there is such a notion, namely that of a group
$G$ acting properly and semisimply on $X$. As shown in \cite{bh}
III.$\Gamma$.1.1, we obtain strong information about $G$ in this case. First
we have a ``small subgroups'' condition, placing a restriction on the virtually
soluble subgroups of $G$ that can occur. Ideally this will say that all
subgroups of $G$ in a certain class (say virtually soluble, virtually
nilpotent or virtually polycyclic) are in fact virtually abelian and
finitely generated. Next we have a ``distortion'' condition saying that
all finitely generated abelian subgroups of $G$ are (if $G$ is itself
finitely generated) undistorted in $G$. Finally there is a
``central elements'' condition, saying that if an element $g\in G$ is
central and of infinite order then, up to finite index, the image of $g$ also
has infinite
order in the abelianisation $G/G'$. (If $G$ is finitely generated then this
is sometimes referred to as ``centralisers virtually split''.) This last
condition is especially useful for showing that some groups do not possess
such an action (for instance see \cite{bri}).

We would like a class of groups, closed under taking arbitrary subgroups
and also finite index supergroups, for which our three conditions above
hold in some appropriate form.
However we would like this class to be defined purely in terms of the
existence of a suitable action on a well behaved class of geometric spaces.
Moreover we will want it to be a wide ranging class, containing all
hyperbolic groups as well as other groups which are well behaved
geometrically such as $\Z^n$ for any $n$. At this stage, if we look back
at the hyperbolic cone construction we note that one feature
is that all infinite order elements act parabolically. Given that the
action of a parabolic element is not as nice geometrically as a
loxodromic element, one could consider the following class: all groups which
act properly and purely loxodromically (every infinite order element acts
loxodromically) on a hyperbolic space, or equivalently by the process of
``graphification'' on a hyperbolic (though not necessarily locally
finite) graph. This is a useful start but it will not contain $\Z^2$. (In fact,
it depends on which definition of a proper action is used, as is discussed
in the example from Section 3. Here though we are using the notion of a
metrically proper action, which $\Z^2$ will not possess.)

A good way round this is to consider proper isometric actions which are
purely loxodromic on a finite product $P$ of hyperbolic graphs
(here a loxodromic element is one whose orbits are quasi-isometrically embedded
once we have put the standard metric on the product). This certainly
contains all hyperbolic groups and $\Z^n$, is closed under taking subgroups
and direct products, and is also closed under taking
finite index supergroups by the technique of induction (if $H$ has
index $n$ in $G$ then this turns an action of $H$ on $P$ into an action
of $G$ on $P^n$ with the same properties). We can then ask: do groups with
such an action satisfy our three
conditions? Unfortunately it satisfies none of them. We will see this for
the first two in Section 7 but the centralisers property can be seen to
fail in a striking way because of the results of \cite{befu} and \cite{bebf}
which imply that the mapping class groups of closed orientable surfaces
lie in this class. In fact a stronger result is obtained: the hyperbolic
graphs are each quasi-isometric to a simplicial tree, and the
group quasi-isometrically embeds via the orbit map of the action in this
finite product of quasi-trees.

However we could ask whether the resulting graphs in this construction
are, or could be made, locally finite. We show that the answer is no and
that mapping class groups of genus 3 or above do not act properly and
purely loxodromically by isometries on any finite product of locally finite
hyperbolic graphs. As a result of Manning in \cite{man} states that
isometries of a quasitree are either elliptic or loxodromic, we deduce
in Section 9 that for any isometric action of the mapping class group of
genus at least three on a locally finite quasitree, all Dehn twists act
as elliptic elements. This also holds for any isometric action
on a bounded valence hyperbolic graph for the same reasons. Moreover we
obtain the same statements for $Aut(F_n)$ and $Out(F_n)$ when $n\geq 4$
(with Dehn twists replaced by Nielsen automorphisms). 

The structure of this paper is that Section 2 briefly mentions some standard
facts about the classification of elements acting isometrically on a
hyperbolic space (here this is a geodesic space, which need not
be a proper metric space, satisfying any of the
standard definitions of hyperbolicity). In Section 3 we review the
various notions of a proper action that do occur, so as to put in
context our definition of being metrically proper and purely loxodromic.
This sets us up for our class of groups considered
in the rest of the paper, namely those with such an action on a finite
product of locally finite hyperbolic graphs (where each factor graph has
the path metric and the $\ell_2$ metric is put on the resulting product).

However in Section 4 we need to examine a potential ambiguity. When
we talk about an isometric action on a finite product $P$ of metric graphs,
do we assume that the action is by automophisms of $P$, so that
it preserves the product graph structure, or should we take the view that
once the metric has been defined on $P$ then any isometry of $P$ is
allowed? We utilise the result in \cite{foly} to explain this
discrepancy, so that any group acting isometrically on $P$ is virtually
a group which preserves factor graphs and acts by graph automorphisms
on each factor, along with an isometric action on the de Rham factor $\R^n$.
This allows us to define our two classes of groups that we consider in detail:
we let $\mathcal M$  be the class acting properly and purely loxodromically
on a finite product of locally finite hyperbolic graphs where we allow
arbitrary isometries and $\mathcal M_0$ is the subclass where each isometry
has to be an automorphism of the product graph structure on $P$.

We show in Section 5 that both $\mathcal M$ and $\mathcal M_0$ have good
closure properties and that the requirement in the definition to act
properly can be ignored. 
However, when considering whether the three conditions
of ``small subgroups'', ``distortion'' and ``central elements'' hold for
all groups in $\mathcal M$ and $\mathcal M_0$, the difference in these two
classes really does matter. We show in Section 7 that all finitely
generated abelian subgroups of a finitely generated group $G_0$ in
$\mathcal M_0$ are undistorted, but point out counterexamples for
$\mathcal M$. As for small subgroups, one variation from the usual
statements along these lines is that asking for a purely loxodromic
action places no requirement on the finite order elements, so all torsion
groups lie in our two classes by default. With that in mind, the small
subgroups statement in Theorem \ref{solm} for $\mathcal M_0$
is as strong as it could be:
any virtually soluble and virtually torsion free subgroup of a group in
$\mathcal M_0$ is finitely generated and virtually abelian. Again this
does not hold in $\mathcal M$.

These statements are deduced from Section 6 where we examine actions
by automorphisms of groups with a loxodromic central element on a locally
finite hyperbolic graph. This allows us to build up a picture in Theorem
\ref{abl} of how the abelian subgroups of a group in $\mathcal M_0$ act.
However the main application of Section 6 is to examine the
``central elements''
condition which we show this time holds for $\mathcal M$ and not just for
$\mathcal M_0$. Specifically, Corollary \ref{mapc} states that if we
have an infinite order central element $g$ of a group $G$ in
$\mathcal M$
then there is a finite index subgroup $L$ of $G$ and $n>0$ such that
the element $g^n$ is in $L$ and has infinite order in the abelianisation
$L/L'$.

This is fine up to finite index, but we then look purely group theoretically
to show that dropping to a finite index subgroup is never needed.
Indeed Corollary \ref{infa} states: 
Suppose that a group $G$ has a finite index subgroup $H$ and an infinite
order element $h\in H$ which is central in $G$. If $h$ has infinite order
in the abelianisation of $H$ then $h$ has infinite order in the abelianisation
of $G$. This is useful for our applications in the remainder of this section
for the following reason. Say we wish to rule out mapping class groups
or $Aut(F_n)$ or $Out(F_n)$ lying in $\mathcal M$. It seems we can take
$G$ to be not the whole group but the centraliser of a Dehn twist/ Nielsen
automorphism $\gamma$, where we know that $\gamma$ has finite order
in the abelianisation of $G$ (by \cite{bri} and by \cite{rwad}). However
in order to apply \ref{mapc} directly, we would need to know what (a power
of) $\gamma$ does in the abelianisation of every finite index subgroup of
its centraliser $G$. This can be done in the case of mapping class groups,
as mentioned in Corollary \ref{mcg}, but our argument allows the immediate
application of the results in \cite{rwad} to show that
$Aut(F_n)$ and $Out(F_n)$ do not lie in $\mathcal M$, let alone $\mathcal M_0$,
when $n\geq 4$.

Finally in Section 9 we look at the case of when the action is on a product
of locally finite graphs which are either quasi-trees or of bounded valence.
Here, as we do not get parabolic elements in these actions, we can convert
our result about mapping class groups, 
$Aut(F_n)$ and $Out(F_n)$ failing to have purely loxodromic actions to the
stronger statement that all Dehn twists and Nielsen automorphisms must act
elliptically. By using our earlier work in Section 8 which covers the
virtual case, we can also say this holds for (powers of) Dehn twists and
Nielsen automorphisms in any finite index subgroup of these 
mapping class groups, as well as $Aut(F_n)$ and $Out(F_n)$ for $n\geq 4$.

\section{Preliminaries}

Suppose that an arbitrary group $G$ acts on an arbitrary metric space
$(X,d)$ by isometries. Given any $g\in G$, it is standard to say that
the element $g$ is {\bf elliptic} under the given action if
the subgroup $\langle g\rangle$ has a bounded orbit (equivalently
all orbits are bounded), which is always the case if $g$ has
finite order. It is also standard to say that 
the element $g$ is {\bf loxodromic} under the given action (though
this can also be called {\bf hyperbolic}) if
the subgroup $\langle g\rangle$ embeds quasi-isometrically in $X$
under the (or rather an) orbit map.
This condition can be expressed in various equivalent ways, but our
definition here will be:
\hfill\\
For a point $x_0\in X$, let the {\bf stable translation length} be
\[\tau(g):=\mbox{lim}_{n\rightarrow\infty} \frac{d(g^n(x_0),x_0)}{n}.\]
This limit always exists and is independent of $x_0$, so we say that
$g$ is loxodromic if $\tau(g)>0$. Indeed
$\tau(g)$ is always equal to the infimum of this sequence,
so that $g$ being loxodromic is equivalent to saying that
we have $c>0$ independent of $x\in X$ with $d(g^n(x),x)\geq nc$
for all $n\in\N$ (equivalently $d(g^n(x),x)\geq |n|c$
for all $n\in\Z$).

This leaves elements $g\in G$ which act with unbounded orbits but
where $\tau(g)=0$. At this level of generality, it
does not seem standard to call these {\bf parabolic} elements.
(For instance we can have an infinite cyclic group $\langle g\rangle$ acting 
isometrically on a metric space $X$ with $x_0\in X$ where
$\{g^n(x_0):n\in\Z\}$ is not bounded but with infinitely many $n\in\Z$
such that $d(x_0,g^n(x_0))=1$.)
We note here
though that an element $g\in G$ is elliptic/loxodromic/``parabolic'' if
and only if $g^n$ is (for some/all $n\in\Z\setminus \{0\}$) and if and
only if some conjugate of $g$ is.

However we do use the term parabolic when $X$ is a hyperbolic space: here
this will always mean a geodesic metric space, satisfying any of
the equivalent definitions of $\delta$ - hyperbolicity, but with no further
conditions such as properness assumed. In the hyperbolic
case we can look at the
action of $G$ by homeomorphisms (though not isometries)
on the boundary $\partial X$ of $X$ to obtain the limit
set $\partial_G X$ which is a subset of $\partial X$. This subset is
$G$-invariant and we have $\partial_H X\subseteq \partial_G X$ if $H$ is
a subgroup of $G$. In particular we can take an arbitrary isometry $g$
of some hyperbolic space $X$ and consider $\partial_{\langle g\rangle} X$.
We can summarise the facts we will need in the following well known proposition.
\begin{prop} \label{cyc}
(i) If $g$ is elliptic then $\partial_{\langle g\rangle} X=\emptyset$. This
is always the case if $g$ has finite order but might also occur for elements
with infinite order. In either case $g$ might fix many or no points on
$\partial X$.\\
(ii) If $g$ is loxodromic then $\partial_{\langle g\rangle} X$ consists
of exactly 2 points $\{g^\pm\}$ for any $x\in X$ 
and this is the fixed point set of $g$ (and $g^n$
for $n\neq 0$) on $\partial X$. Moreover for any $x\in X$ we have
$g^n(x)\rightarrow g^+$ and $g^{-n}(x)\rightarrow g^-$ in $X\cup\partial X$.\\
(iii) If $g$ is parabolic then $\partial_{\langle g\rangle} X$ consists
of exactly 1 point and again this is the fixed point set of $g$ (and $g^n$
for $n\neq 0$) on $\partial X$.\\
\end{prop}

\section{Various notions of proper actions}
Having considered individual isometries, we now 
let $G$ be any group acting on a metric space $X$ by isometries and
we will consider various notions of what it means to say that $G$ acts
properly on $X$.
We will generally assume that $X$ is a geodesic metric space, but place
no further restrictions on $X$ at the moment.

A common notion of a proper action, which we will refer to as
{\bf topologically proper} is that for every compact
subset $C$ of $X$, the set of elements $\{g\in G: g(C)\cap C\neq\emptyset\}$ 
is finite. Alternatively 
we say that $G$ acts {\bf metrically properly} if the above holds
on replacing compact with closed and bounded.
For general metric spaces
acting metrically properly, which can
also be expressed as: for any $x\in X$ and for any $R>0$, the set
$\{g\in G: d(g(x),x)\leq R\}$ is finite, is often considerably
stronger than acting topologically properly. However our particular
results will be for proper metric spaces (meaning that all closed balls are
compact) where these two notions are clearly equivalent.

We also have the concept of a geometric action, which we can think of
as a stronger version of a proper action. Here
we define {\bf cocompact} to mean that there
is a compact subset $C$ whose translates by the group $G$ under the given
action cover all of $X$, and {\bf cobounded} to mean the same but for
a bounded subset $B$. Clearly being cocompact is a stronger notion than
being cobounded in general but is the same in a proper metric space.
We will use the common definition of a {\bf geometric} action, that is
an isometric action which is both topologically
proper and cocompact. However it can sometimes be used to mean that the action
is metrically proper and cobounded. In fact if $X$ is a geodesic metric
space then our notion of a geometric action implies this one,
because $X$ will then be a proper metric space
(see \cite{bh} I.8.4 Exercise 1) and the two definitions are the same anyway
if $X$ is proper.
(If though we remove the geodesic condition from $X$, an infinite group
equipped with the discrete metric
and acting on itself by left multiplication would provide a topologically
proper and cocompact action which is not metrically proper and cobounded.)
Of course these types of geometric actions are important for finitely
generated groups because
the Svarc - Milnor lemma tells us that $G$ and $X$ are quasi-isometric
via the (or a) {\bf orbit map}, namely on equipping $G$ with the word
metric with respect to some finite generating set and
taking any base point $x_0\in X$, we send $g\in G$ to $g(x_0)\in X$.
We present this famous result in the following general form.
\begin{prop} \label{sm}
Let $G$ be an arbitrary group acting on some geodesic metric
space $X$ by isometries. If this action is metrically proper and
cobounded then $G$ is finitely generated and the orbit map is a quasi-isometry
from $G$ to $X$ (where $G$ has the word metric obtained from any finite
generating set).

Conversely if a finitely generated group $G$ acts on $X$ by isometries and
the orbit map is a quasi-isometry then this action is metrically proper
and cobounded (in this direction $X$ can be an arbitrary metric space, not
necessarily geodesic).
\end{prop}

Now it might be that there exists a quasi-isometry between
$G$ and some other metric space $X$ which is not obtained from an action
of $G$ on $X$. (For instance, the free product $C_5*C_5$ is quasi-isometric to
the regular tree of valence 4 but any action of $C_5*C_5$ by isometries on
this tree must be the identity, so that every orbit map is constant.)
However in this paper all quasi-isometries (and quasi-isometric embeddings)
come from an orbit map obtained by an action of a group on a space by
isometries. Therefore we make the following definition:
\begin{defn} An action of a finitely generated
group $G$ on a geodesic metric space $X$ 
is said to be {\bf quasi-isometric} (or QI for short)
if it is metrically proper and cobounded.
\end{defn}

We also define the following notion, which is weaker than a quasi-isometric
action but still pretty strong.
\begin{defn}
An action of a finitely generated group $G$ on a geodesic metric space $X$ 
is said to be a {\bf quasi-isometric embedding} (or a QIE action for short)
if the orbit map is a quasi-isometric embedding,
namely if $||g||$ is the word length of $g\in G$ under our chosen finite
generating set and $x_0\in X$, we have $K\geq 1,C\geq 0$ such that
\[\frac{1}{K}||g||-C\leq d(g(x_0),x_0)\leq K||g||+C.\]
\end{defn}
The above will hold regardless of generating set or basepoint $x_0$
but with possibly different constants. Moreover the second inequality
holds for any isometric action.
Of course here the difference with a QI action is that we might not now
have all points of $X$
within bounded distance of the orbit of $x_0$ under $G$.

We now have five notions of a ``proper/geometric'' action, ranging in strength
from geometric to 
QI to QIE to metrically proper to topologically proper.
(QIE implies metrically proper because there are only finitely
many elements in $G$ with word length at most a given value.)
Thus we should ask: which version of proper should we use in order to
investigate properties of various groups?
Although acting topologically properly clearly has good subgroup closure
properties, it is generally too weak a notion to act as a discriminator for
finitely generated groups, at least if we are allowing actions on spaces
that are not proper. To illustrate this, note that if $\Gamma$ is a
graph (all graphs are assumed to be metric spaces via the path metric
with unit edge lengths)
and a group $G$ acts by graph automorphisms, hence isometries, on
$\Gamma$ then acting topologically properly is equivalent to all vertex
stabilisers being finite (for instance, see \cite{metr} Proposition
4.1 which is a proof for (finite products of)
simplicial trees but which goes through exactly for
(finite products of) arbitrary graphs). Therefore any
infinite group $G$ acts topologically properly by left multiplication
(and indeed coboundedly, but not cocompactly) on the complete graph with
vertex set $G$. As this graph is bounded, it is even hyperbolic. Moreover
note that all elements will act elliptically, whereas we would expect that
under most definitions of a proper/geometric
action, an infinite order element should
not have a bounded orbit and indeed this holds for the four other definitions.

So it seems as if we should use metrically proper as our correct definition,
but every finitely generated group acts metrically properly and freely on a
locally finite hyperbolic graph: namely its hyperbolic cone (as shown in
\cite{grman}; see \cite{osac} Section 3 for a summary of this construction).
Thus even for such nice spaces this notion
fails to distinguish any groups. On the other hand, the class of groups which
act geometrically or indeed QI on some hyperbolic space $X$ are of course
just the word hyperbolic groups by Svarc - Milnor.
Moreover the same is true of QIE actions
(say by \cite{bh} III.H.1.9 as the Cayley graph embeds quasi-isometrically
in $X$).
So in these two cases our permitted groups vary wildly between everything
and a very restricted class. However, in the
example of the hyperbolic cone over a finitely generated group $G$, we are
able to make the action metrically proper at the expense of every infinite
order element acting parabolically. With this in mind we will be looking at
actions where every infinite cyclic subgroup is quasi-isometrically
embedded, rather than the whole group in the case of a QIE action.
Note that a necessary condition for a finitely generated
group $G$ to have such an action on some metric space is that
all infinite cyclic subgroups are undistorted in $G$. Moreover
if all infinite cyclic subgroups of $G$ are undistorted then the action
of $G$ on its Cayley graph is such an action on a proper geodesic
metric space.

The class of finitely generated groups with undistorted infinite cyclic
subgroups is wide and might be thought of as a reasonable initial
definition of groups that are not badly behaved geometrically. Moreover
this class is preserved under taking finitely generated subgroups and
if $G$ is in this class then a QIE action of $G$ on some metric space $X$
implies that the action is {\bf purely loxodromic}; that is
every infinite order element acts loxodromically.
So in line with the examples above, we can move from specific actions to
groups possessing such an action on a suitable space and ask: what is the
class of groups admitting a proper isometric and purely loxodromic action
on some hyperbolic space, or some hyperbolic graph, or even some locally finite
hyperbolic graph? These classes clearly
include all hyperbolic groups and (using the hyperbolic cone)
all countable torsion groups,
but also all subgroups of hyperbolic groups, thus they are somewhat wider
than the class of groups admitting a QIE action on a hyperbolic space which
was just the hyperbolic groups. However the difference between requiring
or not requiring that the action is also metrically proper
really matters as soon as we consider $\Z\times\Z$.\\
\hfill\\
{\bf Example}:
The action of $\Z\times\Z$ on $\R$ via $x\mapsto x+1,x\mapsto
x+\sqrt{2}$ is purely loxodromic and indeed free,
but this is surely not what would
be regarded as a proper action in any sense. Indeed
there is no metrically proper action of 
$\Z\times\Z$ on a hyperbolic space in which every
non-identity element acts loxodromically (for instance see Proposition
\ref{zprop}).
However we can form
a topologically proper and isometric action of
$\Z\times\Z$ on a hyperbolic graph which is also
purely loxodromic: namely take the above action on $\R$
(or even on $\Z\oplus\Z[\sqrt 2]$) and let the vertex set of $\Gamma$
be the points of this space, with an edge joining vertices if they are
at most distance 1 away. Then $\Z\times\Z$ acts by automorphisms on
$\Gamma$, and thus by isometries once we equip $\Gamma$ with the
standard path metric. This ``graphification'' process also preserves
the elliptic/loxodromic/``parabolic'' classification of the original
action of $G$. Moreover
$\Gamma$ is quasi-isometric to our original
space, namely $\R$, and the action is still free which means that it is
also topologically proper (by \cite{metr} Proposition 4.1).

Now we have agreed that the notion of acting topologically properly
is too weak in general and therefore a proper action will mean
metrically proper from now on. But we certainly do not wish to discard
well behaved groups just because they contain $\Z^2$.
Thus in order to obtain a suitably wide ranging class of groups with
well behaved actions which includes $\Z^2$ (and indeed
$\Z^n$ for any $n$) as well as all hyperbolic groups and their
subgroups, we now look at
groups acting in this way on finite products of hyperbolic spaces, where
we can take these spaces to be graphs which are hyperbolic with respect to
the path metric.
We can then distinguish between the case where all graphs are locally finite
and the case where arbitrary hyperbolic graphs are allowed. We will see that
groups possessing such an action in the locally finite case are much better
behaved overall than in the general case.

\section{Groups acting on products of graphs}

It is clear that if $\Gamma$ is a connected graph which is given
the standard path metric then the isometry group of $\Gamma$ is precisely
the group $Aut(\Gamma)$ of graph automorphisms of $\Gamma$, with the one
exception where $\Gamma$ is the simplicial line $L$, in which case we have
$Isom(L)=Isom(\R)$ and $Aut(L)=Isom(\Z)$. Let us now take a
finite number of connected
graphs $X=\Gamma_1\times \ldots \times \Gamma_k$, each
equipped with their own path metric $d_i$, and turn $X$ into a metric space
using the $l_2$ metric. Hence for ${\bf x}=(x_1,\ldots ,x_k)$ and
${\bf y}=(y_1,\ldots ,y_k)\in X$ we have
\[d({\bf x},{\bf y})^2=\sum_{i=1}^kd_i(x_i,y_i)^2.\]
Consequently $X$ is a geodesic metric space because each factor is.
However we can also use the graph structure on each $\Gamma_i$ to turn
$X$ into a $k$ dimensional cube complex, where a vertex of $X$ is
$(v_1,\ldots ,v_k)$ for $v_i$ a vertex of $\Gamma_i$, edges in $X$ join  
$(v_1,\ldots,v_i,\ldots,v_k)$ to $(v_1,\ldots,w_i,\ldots,v_k)$ for some $i$,
where $v_i-w_i$ is an edge in the graph $\Gamma_i$, and so on. As we are using
the $l_2$ metric for our product, this metric on $X$ is the same as
putting the standard Euclidean metric on each of the $i$ dimensional cubes
and then forming the path metric using the cube complex structure. This
means that any automorphism of $X$ with the cube complex structure will
be an isometry. One way to obtain such isometries is as follows.
\begin{defn} If $X=\Gamma_1\times \ldots \times\Gamma_k$ then we say that a
group $G$ acts {\it preserving factors} if $G$ acts separately by graph
automorphisms on each factor $\Gamma_i$.
\end{defn}
It is then clear that $G$ will both preserve this
cube complex structure and will act by isometries on $X$, but the
advantage here is that we also have a homomorphism of $G$ to
$Aut(\Gamma_1)\times \ldots \times Aut(\Gamma_k)$ and thus to
each $Aut(\Gamma_i)$ separately by projection. In particular any element
$g\in G$ can be written $g=(g_1,\dots ,g_k)$ for $g_i\in Aut(\Gamma_i)$
and we can then ask whether each $g_i$ is an elliptic, ``parabolic'' or
loxodromic isometry with respect to its action on $\Gamma_i$.
This means that, although $X$ itself is almost never a hyperbolic space
even if the factors are, in
the case where $G$ acts on $X$ preserving factors, we can relate the type
of $g$ as an isometry of $X$ to the type of its action on the factors
(see Proposition \ref{type} below).

However $G$ could still act as a group of automorphisms of the cube complex
and thus by isometries of $X$ if it permutes some of the graphs.
Of course in this case it can only permute graphs which are isomorphic, but
in any event we will always have a finite index subgroup $H$ of $G$
which preserves factors and in many cases it is enough to work with $H$
instead of $G$.

We now need to consider the reverse process. As all isometries of $X$
are potentially of interest to us, we must ask whether every isometry
of $X$ is an automorphism of the cube complex structure (in which
case some power will preserve factors)? We can see that
the answer is no in general, just as in the case for a single graph: if
we take $k$ copies of the simplicial line $L$ then
$L^k$ with its natural cube complex structure has an automorphism group
which is virtually a copy of $\Z^k$, whereas $Isom(L^k)=Isom(\R^k)$ is
a much bigger group. This is a classical phenomenon in Riemannian geometry
dealt with by the de Rham theorem, but the wide ranging generalisations of
this in \cite{foly} allow us to apply their results directly to our case to
see that this is the only exception. In particular from that paper we have
the following, where a metric space $X$ is said to be irreducible if whenever
we have $X=Y\times Z$ for $Y,Z$ metric spaces under
the $\ell_2$
product metric then either $Y$ or $Z$ is a point. 
\begin{thm} \label{prod} (\cite{foly} Theorem 1.1 and Corollary 1.3)
Let $X$ be a geodesic metric space of finite affine rank. Then $X$ admits
a unique decomposition as a metric direct product
\[X=Y_0\times Y_1\times\ldots \times Y_n\]
where $Y_0$ is an isometric Euclidean copy of $\R^m$ (or a point) and each 
$Y_i$ is an irreducible metric space not isometric to $\R$ or a point.

In particular, if $\cal P$ is the group of all permutations
$\sigma\in \Sigma_n$ such that $Y_i$ and $Y_{\sigma_i}$ are isometric
for each $1,\ldots ,n$ then we have the  natural exact sequence
\[1\rightarrow Isom(Y_0)\times Isom(Y_1)\times \ldots \times Isom(Y_n)
\rightarrow Isom(X)\rightarrow{\mathcal P}\rightarrow 1.\]
\end{thm}

Here the affine rank of $X$ is the supremum over all topological dimensions 
of affine spaces which isometrically embed into $X$ and is bounded above
by the topological dimension of $X$. In our case where
$X=\Gamma_1\times\ldots \times\Gamma_k$, we have that the affine rank
of $X$ is therefore equal to $k$ which is its topological dimension.
Also by noting that if $X=Y\times Z$ as metric spaces then $X$ is a geodesic
metric space if and only if $Y$ and $Z$ both are, we see that a graph
is an irreducible metric space as it does not contain an embedded copy
of a square. Thus in our situation where $X=\Gamma_1\times\ldots\times\Gamma_k$,
we see that if no $\Gamma_i$ is equal to the simplicial line $L$ then this
decomposition is the above decomposition into irreducible metric spaces, and
otherwise we group all copies of $L$ together to get this decomposition.
Thus we obtain directly from this discussion and Theorem \ref{prod}:
\begin{co} \label{ismprd}
Suppose that $X=\Gamma_1\times \ldots \times \Gamma_k$ is a finite product of
graphs, where each $\Gamma_i$ has the induced path metric and $X$
has the $\ell_2$ product metric. Suppose that $G$ is any group
acting by isometries on $X$.\\
(i) If no graph $\Gamma_i$ is the simplicial line then $G$
has a finite index subgroup $H$ which preserves the
factors and acts by graph automorphisms on each factor $\Gamma_i$.\\
(ii) If $\Gamma_{j+1},\ldots,\Gamma_k$ are those factors which are the
simplicial line $L$ then $G$ has a finite index subgroup $H$ which
preserves each factor in the decomposition
$\Gamma_1\times \ldots \times \Gamma_j\times \R^{k-j}$ where $H$ acts
on each $\Gamma_1,\ldots ,\Gamma_j$ by graph automorphisms and by
Euclidean isometries on $\R^{k-j}$.
\end{co}

We can now relate all of this to the classification of isometries into the
elliptic/``parabolic''/loxodromic cases. We first note that if $g$ is an
isometry of $\R^n$ then $g$ can only be elliptic or loxodromic (see for
instance \cite{bh} II.6.5) as the ``parabolic'' case cannot occur.
\begin{prop} \label{type}
Suppose that $X=\Gamma_1\times \ldots \times\Gamma_k$ is a product of
graphs and $g$ is an isometry of $X$. If $h=g^n$ is an
isometry preserving the decomposition $\Gamma_1\times \ldots \times \Gamma_k$
or $\Gamma_1\times \ldots \times \Gamma_j\times \R^{k-j}$ as guaranteed by
the above result then we have:\\
(i) The isometry $g$ acts elliptically on $X$ if and only if $h$ acts
elliptically on each of the factors.\\
(ii) The isometry $g$ acts ``parabolically'' on $X$ if and only if $h$ acts
elliptically or parabolically on each of the factors, with at least
one action on these factors being parabolic.\\
(iii) The isometry $g$ acts loxodromically on $X$ if and only if at least one
of the actions of $h$ on the factors is loxodromic.
\end{prop}
\begin{proof}
First of all we can replace $g$ with $h$ as this will not affect the type
of action on the whole of $X$. Case (i) holds because we will have bounded
orbits on all factors if and only if there is a bounded orbit on the whole
space. Similarly say $h$ acts as a loxodromic element $h_i$ on $\Gamma_i$,
so there is $c>0$ such that for an arbitrary vertex $v_i\in \Gamma_i$ we have
$d_i(h_i^m(v_i),v_i))\geq c|m|$. Then for any vertex $v\in X$ we
have $d(h^m(v),v)\geq d_i(h_i^m(v_i),v_i))\geq c|m|$ so that $h$ is loxodromic
too. Thus we are only left with the case where $h$ acts as a parabolic
element $h_i$ on some $\Gamma_i$ but where no action on a factor is
loxodromic. Again using the inequality $d(h^m(v),v)\geq d_i(h_i^m(v_i),v_i))$,
we see that $h_i$ having unbounded orbits implies the same for $h$. However
in each factor $\Gamma_j$ 
we will also have $d_j(h_j^m(v_j),v_j)/m$ tending to zero as $m$ tends to
infinity, so by adding and using inequalities we have that
$d(h^m(v),v)/m$ tends to zero too, thus $h$ does not act loxodromically
on $X$.
\end{proof}

\section{Proper actions on products of graphs}

We saw in Section 3 a range of definitions of what it might mean for a
group to act properly/geometrically on a geodesic metric space, with particular
reference to the case where $X$ was a hyperbolic space and even
a hyperbolic graph where we allowed both locally finite and non locally
finite graphs. Here we consider how these definitions work when our space
is a finite product of locally finite hyperbolic graphs.
As this will result in a proper metric space, metrically  and topologically
proper actions are the same. In particular we 
now introduce the two classes of groups that we will be studying.
\begin{defn} The class $\mathcal M$ is the class of groups possessing
a purely loxodromic and (metrically) proper
action by isometries on a finite product of locally finite
hyperbolic graphs.

The subclass $\mathcal M_0$ is the class of groups with
such an action by automorphisms.
\end{defn}

If our focus
is on countable groups (which is the case throughout this paper anyway)
then in fact the metrically proper part of this definition becomes redundant.
\begin{prop} \label{nomp}
Suppose that $G$ is a group having an action by isometries/automorphisms on a
finite product $P$ of locally finite hyperbolic graphs which is purely
loxodromic. Then the group $G$ has a purely loxodromic action by
isometries/automorphisms on such a space
which is also metrically proper
if and only if $G$ is countable.
\end{prop}
\begin{proof} Any group with a metrically proper action on some metric
space must be countable. As for the converse, if $G$ is countable and has
an action by isometries/automorphisms on $P$ where every infinite order
element is loxodromic then take the metrically proper action of $G$ on its
hyperbolic cone $X$, which is a locally finite hyperbolic graph. Then $G$
acts by isometries/automorphisms on $P\times X$. Moreover for any $g\in G$
and $(p,x)\in P\times X$ we have $d_{P\times X}(g(p,x),(p,x))\geq
d_X(g(x),x)$, thus this action by isometries/automorphisms on $P\times X$
is metrically proper too.
\end{proof}
Thus to show a given countable
group $G$ is in our class $\mathcal M_0$ (or $\mathcal M$) it is enough
to find a finite index subgroup $H$ of $G$ and an action of $H$ by
automorphisms (or by isometries)
on a finite product of locally finite hyperbolic graphs
preserving factors, where every infinite order element of $H$ acts
loxodromically on one of the factors, by Proposition \ref{type}.

Our two classes of groups introduced above have good
closure properties as we now record.
\begin{prop} Both of the classes $\mathcal M$ and $\mathcal M_0$
are closed under taking finite direct products, arbitrary subgroups
and supergroups of finite index.
\end{prop}
\begin{proof} These operations clearly preserve countability of the groups
involved.

If the group $A$ has an action by isometries/automorphisms
on a product of locally finite hyperbolic graphs $P$ which is
purely loxodromic and $B$ has
such an action on $Q$ then we will have the direct product action
on $P\times Q$ (with the obvious $\ell_2$ product metric) via
$(a,b)((p,q))=(a(p),b(q))$. Suppose that some element $(a,b)$ has
infinite order but does not act loxodromically on $P\times Q$, which means
that neither $a$ on $P$ nor $b$ on $Q$ acts loxodromically. Hence
we conclude that both $a$ and $b$ have finite
order and so does $(a,b)$. Otherwise $a$ (say) acts loxodromically on $P$,
giving us $c>0$ such that
\[c\leq \frac{d_P(a^n(p),p)}{n}\leq \frac{d_{P\times Q}((a,b)^n((p,q),(p,q))}
{n}\]
and thus $(a,b)$ is loxodromic in the product action.

Subgroup closure is clear. If $H$ has index $i$ in $G$, with $H\unlhd G$ without
loss of generality, and $H$ acts suitably on the product $P$ then the induced
action of $G$ is an action by isometries/automorphisms
on $P^i$ which permutes these factors. The
corresponding induced action of $H$ on $P^i$ preserves these factors and
is the original action on the first $P$ factor of $P^i$ and a conjugate
action on the other factors. Thus any infinite order element $g\in G$ has
$g^n\in H$ which is acting loxodromically on $P^i$ because it does so on
each factor.
\end{proof}

The class of groups in $\mathcal M$ and in $\mathcal M_0$ is wide:
for instance they both contain (as well as subgroups, finite index supergroups
and direct products of the following) $\Z^n$ for any $n$, finite direct
products of arbitrary hyperbolic groups and their subgroups, as well
as Burger - Mozes groups (these act geometrically on
products of locally finite trees).
Also included are all finitely generated subgroups of $PGL_2(\F_q(t))$ for
$q$ a power of a prime, as we can take a discrete
valuation on $\F_q(t)$ and the
corresponding action on the Bruhat - Tits tree which is regular of valence
$q+1$. If the group is finitely generated then we can take a finite number
of valuations such that the stabiliser of any vertex in this product of
trees is finite (for instance see \cite{meggd} Section 2).
This includes interesting examples such as the wreath
product for a prime $p$
\[C_p\wr \Z=\langle \sma{cc}1&1\\0&1\fma, \sma{cc}t&0\\0&1\fma\rangle
\mbox{ in }PGL_2(\F_p(t)).\]

\section{Locally finite hyperbolic graphs}
In the next three sections we look for good group theoretic properties which
are possessed by all of $\mathcal M_0$ (though it will turn out that some
of these properties hold for all groups in $\mathcal M$ and some do not).
The properties we examine are akin to ones shared by all word hyperbolic
and non positively curved groups. Specifically there should be a result
on ``small subgroups'' which gives restrictions on which virtually soluble
groups can occur, these ``small subgroups'', if finitely generated, should
be undistorted in any finitely generated group of $\mathcal M_0$, and
(most importantly for applications) there should be some restriction on the
centraliser of an infinite order element in these groups.
Consequently we will need to utilise the hyperbolicity of our factor
graphs, as well as their local finiteness (and in some cases the fact that
the action is by automorphisms and not just isometries).

We first mention some standard definitions and results.
Given a metric space $X$ and constants $K\geq 1,C\geq 0$,
we say that a function $f:\R\rightarrow X$
is a $(K,C)$-{\it quasigeodesic} if for all $s,t\in\R$ we have
\[\frac{1}{K}|s-t|-C\leq d(f(s),f(t))\leq K|s-t|+C.\]
If instead the domain is $\Z$ but the same equation holds then here
we say that the function is a {\it discrete} $(K,C)$-quasigeodesic.
On being given any isometry $g$ and any point $x_0$ of an arbitrary metric
space $X$, we have that the orbit of $x_0$ under $\langle g\rangle$ is a
discrete quasigeodesic if and only if $g$ is a loxodromic element.
Indeed if $g$ is loxodromic then we can take $C=0$ as we saw in Section 2
that
$\tau(g)|m-n| \leq d(g^m(x_0),g^n(x_0))$. Thus the lower bound is
independent of $x_0$, but we also have
\[d(g^m(x_0),g^n(x_0))=d(g^{m-n}(x_0),x_0))\leq |m-n| d(g(x_0),x_0)\]
so that the upper bound will necessarily depend on $x_0$, indeed we
can take $K=\mbox{max}(D_g(x_0),1/\tau(g))$ where $D_g:X\rightarrow [0,\infty)$
is the displacement function of $g$.

Suppose now that $X$ is any geodesic metric space (here geodesics need not be
unique) and we have an isometry $g$ of $X$ which is acting loxodromically.
We can obtain a continuous quasigeodesic from the orbit of $x_0$ by first
joining $x_0$ to $g(x_0)$ by some geodesic segment $\gamma$ of length
$l:=D_g(x_0)>0$ and then for each $n$ in $\Z\setminus\{0\}$ we join $g^n(x_0)$ 
to $g^{n+1}(x_0)$ by the image of $\gamma$ under $g^n$, which is also
a geodesic segment of length $\gamma$.
We can then define $f:\R\rightarrow X$ by sending $[n,n+1]$ evenly to
this geodesic $g^n(\gamma)$, so we first map $[n,n+1]$ to
$[ln,l(n+1)]$ and then map this isometrically to $g^n(\gamma)$.
Now suppose that we have $s,t\in\R$ with $s\in [m,m+1]$ and $t\in [n,n+1]$.
We can then see that
\[l|m-n|-2l\leq d(f(s),f(t))\leq l|m-n|+2l\]
but as we have $|t-s|-1\leq |m-n|\leq |t-s|+1$, we obtain
\[l|s-t|-3l\leq d(f(s),f(t))\leq l|s-t|+3l.\]
In particular if $l\geq 1$ then we have a $(l,3l)$-quasigeodesic
containing the orbit of $x_0$ under $\langle g\rangle$ and whose image is
$\langle g\rangle$-invariant.

So far we have not used anything about $X$ other than it is a geodesic
metric space. But let us now say that $X$ is a $\delta$-hyperbolic space
and $g$ is a loxodromic isometry of $X$ with fixed points $g^\pm$ on
$\partial X$. We know that on
$X\cup\partial X$, it is the case that for any $x\in X$ we have
$g^n(x)\rightarrow g^+$ as $n\rightarrow\infty$ and $g^n(x)\rightarrow g^-$ as
$n\rightarrow -\infty$, so for $l\geq 1$ our $(l,3l)$-quasigeodesic has
endpoints $g^\pm$. We now utilise the stability of quasigeodesics
with the same end points, by quoting a version of this property
that is Lemma 2.11 in \cite{man}.
\begin{lem} \label{mann}
Let $K\geq 1$, $C\geq 0$ and $\delta\geq 0$. Then there is some
$B=B(K,C,\delta)$ so that if $\gamma$ and $\gamma'$ are two
$(K,C)$-quasigeodesics with the same endpoints in $X\cup\partial X$, and
$X$ is a $\delta$-hyperbolic geodesic metric space then the image of
$\gamma$ lies in a $B$-neighbourhood of the image of $\gamma'$.
\end{lem}

Let $X$ be a locally finite graph with vertex set $V(X)$
which is $\delta$-hyperbolic for some $\delta>0$ when turned into
a geodesic metric space (which will be proper),
on being given the path metric with all edges
having length 1. We apply the above result as follows:
suppose that $g$ is a loxodromic isometry of $X$ which is acting as
a graph automorphism of $X$. Then
there are no fixed points and we can take a vertex $v_0\in X$ of
minimum displacement amongst all vertices, so that $l=D_g(v_0)$ is
an integer at least 1. We now define a subset of $V(X)$ that will
act as a discrete version of an axis for us.

\begin{defn}
Let the {\it pseudoaxis} $P_g$ of the isometry $g$ be
\[P_g:=\{v\in V(X): d(v,g(v))=d(v_0,g(v_0))\}.\]
\end{defn}
This set is non empty, as it contains the orbit of $v_0$, and is
$\langle g\rangle$-invariant so is a union of $\langle g\rangle$-orbits.
Note also that for another isometry $h$ of $X$ we have
$P_{hgh^{-1}}=h(P_g)$ and $P_g=P_{g^{-1}}$, although we need not have
$P_{g^n}=P_g$ in general if $n>0$.
\begin{prop} \label{fini}
If $g$ is a loxodromic isometry of a locally finite
$\delta$-hyperbolic graph then the pseudoaxis $P_g$ is a finite union of
$\langle g\rangle$-orbits.
\end{prop}
\begin{proof}
Given a vertex $v_0$ in $P_g$ of minimal displacement $l$, create a
$(l,3l)$-quasigeodesic $\gamma_0$ as above from the orbit of $v_0$
under $\langle g\rangle$. Now let $B$ be the value $B(l,3l,\delta)$
in Lemma \ref{mann}
and take any other vertex $v\in P_g$ which is
not in the orbit of $v_0$ (if there are any, else we are done).
On setting $\gamma$ to be another $(l,3l)$-quasigeodesic, this time
obtained from the orbit of $v$, we have by
Lemma \ref{mann} that there is a point
$x_v$ on $\gamma_0$ with $d(v,x_v)\leq B$ in the path metric on $X$. But
by construction we have $m\in\Z$ such
that $d(x_v,g^m(v_0))\leq l$ which implies
$d(g^{-m}(v),v_0)=d(v,g^m(v_0))\leq B+l$. Now local finiteness of $X$ means
that the set of vertices having distance at most $B+l$ from $v_0$ is finite,
and we have shown that every vertex in the pseudoaxis $P_g$ lies in the same
orbit as some vertex from this finite set. Thus $P_g$ consists of a finite
number of $\langle g\rangle$-orbits.
\end{proof}
We will use this result to examine centralisers of infinite order elements
in our two classes of groups, as has been done for CAT(0) and other
similar groups (for instance see \cite{bh} II.6.12).
\begin{thm} \label{main}
Suppose that we have a group $H$ acting by automorphisms on the locally finite
hyperbolic graph $X$ and a loxodromic element $g\in H$ which is central in
$H$. Then no elements are parabolic and
there is a normal subgroup $L$ of finite index in $H$ which contains
$g$, along with a homomorphism $\theta:L\rightarrow\Z$ such that:\\
(i) we have $\theta(g)=1$ and\\
(ii) the kernel of $\theta$ consists precisely of the elliptic elements
in $L$ under this action.
\end{thm}
\begin{proof}
Our group $H$ acts on the boundary $\partial_H X$ with parabolic elements
having one fixed point there, whereas the loxodromic element $g$ has two.
But a bijection with exactly two fixed points cannot commute with a bijection
that has just one.
For any $h\in H$ we have $hgh^{-1}=g$ so that
$P_g$ is $H$-invariant. Now
Proposition \ref{fini} tells us that there are vertices
$v_0,v_1,\ldots ,v_{k-1}\in V(X)$ with
\[P_g=O_0\sqcup \ldots \sqcup O_{k-1}\]
where $O_i=\mbox{Orb}_{\langle g\rangle}(v_i)$. But $g$ being central means
that this decomposition is preserved under the action of $H$, as if
$v=g^m(v_i)$ for some $m$ and $i$ then for any $h\in H$ we have
$h(v)=hg^m(v_i)=g^mh(v_i)$, so that $v$ and $v_i$ being in the same
$\langle g\rangle$-orbit implies that $h(v)$ and $h(v_i)$ are too. Thus
we have a well defined quotient action of $H$ on the $k$ orbits
$O_0,\ldots ,O_{k-1}$ and thus on setting $L$ to be the kernel of this
action (which certainly contains $g$), we have that $L$ fixes setwise these
orbits, so in particular $L(O_0)=O_0$ where $O_0=\{g^m(v_0):m\in\Z\}$.

If $O_0$ were an isometrically embedded, rather than a quasi-isometrically
embedded, copy of $\Z$ then we could argue that $L$ acts by isometries
and so by translations on $O_0$. Instead we now throw away the metric
altogether: we have $L$ acting by permutations on a copy of $\Z$ with $g$
acting centrally as the shift $\sigma$, where $\sigma(n)=n+1$. Then if we
take any element in $L$ which acts on this copy of $\Z$ by the permutation
$f$ say, on setting $N=f(0)$ we have
\[f(n)=f(\sigma^n(0))=\sigma^n(f(0))=n+N\]
so all elements of $L$ also
act on $\Z$ by translations. This action gives us our homomorphism $\theta$
of $L$ onto $\Z$ sending $g$ to 1, and anything in the kernel of $\theta$
fixes the orbit of $v_0$ pointwise. Whilst these elements need not act as
the identity on the whole of $X$, nevertheless they fix vertices and so
are all elliptic. On the other hand the orbit of $v_0$ under any element
of $L$ which is not in the kernel of $\theta$ is unbounded.
\end{proof}

We might wonder if we can take $L$ equal to $H$ in the above. Although we cannot
expect a homomorphism $\theta$ from $H$ onto $\Z$ with $\theta(g)=1$ because
$g$ could be a proper power of some $h \in H$, we can indeed obtain under
exactly the same hypotheses and by
similar methods a homomorphism $\phi$ from $H$ to $\Z$ with $\phi(g)\neq 0$
and such that the kernel of $\phi$ consists of exactly those elements
of $H$ which
act elliptically on $X$. We delay the proof of this result until Corollary
\ref{impr} because in Section 8 we will deduce it from a much more general
statement. However we will feel free to use this stronger version in the next
section, where the results can all be obtained by applying Theorem \ref{main}
but the application of Corollary \ref{impr} avoids having to drop repeatedly
to finite index subgroups.

Note however that this result is certainly not true if we remove the locally
finite
condition on $X$, as we will see in Section 8. What continues to be true
for $X$ a general hyperbolic space is that we obtain a homogeneous
quasimorphism $q$ from $H$ to $\R$ with $q(g)\neq 0$;
see for instance \cite{abhos} Subsection 4.3 and related references.

We finish this section by pointing out the (folklore) result that there is
no purely loxodromic action of $\Z\times\Z$ on any hyperbolic space which
is also metrically proper (as opposed to the topologically proper example
in Section 3). Indeed there can be no such action with a single
loxodromic element. This result is akin to the fact that the
centraliser of any infinite order element in a hyperbolic group is virtually
cyclic and follows from Lemma \ref{mann}.
\begin{prop} \label{zprop}
Suppose the infinite order element $g$ is central in the group $G$ which
acts metrically properly by isometries on some hyperbolic metric space $X$.
If $g$ acts loxodromically on $X$ then $G$ is virtually cyclic.
\end{prop}
\begin{proof}
We create a quasigeodesic from the $\langle g\rangle$-orbit of any point
$x_0\in X$ as in Section 6 and then, for any $h\in G$, we create another
one starting with $h(x_0)\in X$. By Lemma \ref{mann} we have $C>0$
which is independent of $h$ such that there is $m\in\Z$ (which does
depend on $h$) with
\[d(h(x_0),g^m(x_0))=d(x_0,h^{-1}g^m(x_0))\leq C.\]
Thus on applying the metrically proper condition of the action of $G$ at
$x_0$, there is $g_1,\ldots ,g_k\in G$ such that $h^{-1}g^m$ is equal
to one of these elements. Hence $h\in\langle g\rangle g_1^{-1}\cup
\ldots \cup \langle g\rangle g_k^{-1}$ and so $G$ is virtually cyclic.
\end{proof}

\section{Investigation of small subgroups}

We now consider the abelian and more generally the soluble groups in
$\mathcal M$ and $\mathcal M_0$. The perfect result, as occurs for
hyperbolic and CAT(0) groups, would say that any virtually
soluble subgroup (a ``small'' subgroup) of a group $G$ in $\mathcal M$
or $\mathcal M_0$ (thus here the subgroup would also be in $\mathcal M$
or $\mathcal M_0$) is in fact virtually abelian, finitely generated and
undistorted in $G$ if $G$ is finitely generated. However as all
countable torsion groups are in $\mathcal M$ and $\mathcal M_0$ by
default, we already have a range of abelian and soluble groups in our two
classes which are not finitely generated. Therefore, in first looking
at possible restrictions on abelian groups, our question should
be phrased as: suppose $A$ is an abelian group in $\mathcal M$ or
$\mathcal M_0$ which is torsion free then what can we say
about $A$? In particular is it finitely generated? This is clearly not
true in $\mathcal M$ as this class allows isometric actions which are
not by automorphisms, such as any countable group acting freely
by isometries on $\R^n$ (by Proposition \ref{nomp} because
all fixed point free isometries of $\R^n$ are loxodromic).
As an example, we see that the free abelian group $\Z^\infty$ (a restricted
direct product) is in $\mathcal M$ as it acts freely by isometries on $\R$
by taking infinitely many real numbers $r_1,r_2,\ldots$ that are linearly
independent over $\Q$ and then forming the group
\[\langle x\mapsto x+r_1,x\mapsto x+r_2,\ldots \rangle\]
in which every non-identity element is loxodromic.

However if we stick to $\mathcal M_0$ then we do get the best possible result
for abelian groups.

\begin{thm} \label{abl}
Consider the class of countable abelian groups $A$ which act purely
loxodromically by automorphisms
on some finite product $P$ of locally finite hyperbolic graphs.
If $T$ is the torsion
subgroup of $A$ then $A$ is in this class if and only if $A/T$ is
is finitely generated.
\end{thm}
\begin{proof}
If $A/T$ is finitely generated, thus equal to $\Z^k$ because it is torsion
free, then we can let $A$ act on $\R^k$ by integral translations (thus
by automorphisms on the product of $k$ copies of the simplicial line)
with $T$ acting trivially. We can then take the standard action of $A$
on its hyperbolic cone $X$, whereupon the corresponding action of
$A$ on $\R^k\times X$ shows that $A$ is in $\mathcal M_0$.

Now suppose that $A$ is a countable abelian group which does have
a purely loxodromic 
action by automorphisms on $P=\Gamma_1\times\ldots \times\Gamma_n$ say.
We further assume
that there is some infinite order element in $A$ (else $A/T$ is trivial).
As mentioned in Section 4 we can replace $A$ with a finite
index subgroup $L$ which preserves factors, whereupon we still
have that all infinite order elements are loxodromic and that some
element of $L$ acts loxodromically on this smaller product. Now for
each $1\leq i\leq n$ in turn, we take an element $a_i\in L$ which acts
loxodromically on $\Gamma_i$ but which does not act
loxodromically on all other graphs $\Gamma_j$.
If in this process we find for some $i$ that no such element
$a_i$ exists then we can remove $\Gamma_i$ from the product and consider
the action of $L$ on
\[\Gamma_1\times\ldots\times\Gamma_{i-1}\times\Gamma_{i+1}\times \ldots
\times \Gamma_n\]
which still preserves factors. Moreover we still have that every infinite
order element acts loxodromically and that there is such an element.

Having finished this process, we relabel and renumber to get an action of
$L$ on a new product $Q=\Gamma_1\times\ldots\times\Gamma_m$ for $m\leq n$
where there are infinite order elements $a_1,\ldots ,a_m$ which have the
property that $a_i$ acts loxodromically on $\Gamma_i$ and elliptically
on $\Gamma_j$ for $i\neq j$ (because $a_i$ and $a_j$ commute).
The existence of a 
loxodromic element initially ensures here that $m>0$.

For each $i$ from 1 to $m$
we now apply Corollary \ref{impr} to the group $L$, the central element being
$a_i$ and the graph $X$ equal to $\Gamma_i$. We obtain a
homomorphism $\theta_i:L\rightarrow\Z$ with
$\theta_i(a_i)\neq 0$ and with $\mbox{Ker}(\theta_i)$ consisting of the elements
in $L$ acting elliptically on $\Gamma_i$.
Combining these for $1\leq i\leq m$ gives us
a homomorphism $\theta:=\theta_1\times\ldots\times
\theta_m:L\rightarrow\Z^m$ with image having finite index in $\Z^m$
(because $\langle \theta(a_1),\ldots ,\theta(a_m)\rangle$ has finite index)
and $K:=\mbox{Ker}(\theta)$ consisting of
elements in $L$ which are elliptic on each factor graph. This means that they
will be elliptic elements in the action on the new product $Q$ by
Proposition \ref{type}. As every infinite order element of $L$ acts
loxodromically on our original product $P$ and indeed on the new product
$Q$ by construction,
we have that $K$ consists of finite order elements, but
a finite order element of $L$ would be in $K$ anyway, thus $K=L\cap T$.
Thus $\theta$
is a homomorphism from $L$ to 
$\Z^m$ with image having finite index in $\Z^m$, hence also isomorphic
to $\Z^m$, and
kernel the torsion subgroup $L\cap T$. So we have
$\Z^m\cong L/(L\cap T)\cong LT/T$, but $LT$ is certainly a finite
index subgroup of the original group $A$.
Thus $A/T$ has $LT/T\cong \Z^m$ as a finite index subgroup and thus is
itself finitely generated and (being torsion free)
isomorphic to $\Z^m$.
\end{proof}

We obtain two corollaries on abelian (sub)groups in $\mathcal M_0$
from this result, of which the first is immediate from the above proof.
\begin{co} If $A\cong\Z^m$ acts purely loxodromically
by automorphisms on the product
$\Gamma_1\times\ldots\times\Gamma_n$ of locally finite hyperbolic
graphs then $m\leq n$.
\end{co}

Note that we cannot remove either ``by automorphisms'' or ``locally finite''
in this statement, as evidenced by the group $\Z^\infty$ with the
action above for the former, and the example in Section 3 of $\Z\times\Z$
acting purely loxodromically on a hyperbolic graph for the latter.

We can also look at distortion of finitely generated abelian subgroups.
\begin{co} \label{unds}
If $G$ is a finitely generated group in $\mathcal M_0$ containing a finitely
generated abelian subgroup $A$ then $A$ is undistorted in $G$.
\end{co}
\begin{proof}
It is enough to show this for any finite index subgroup of $A$, thus we can
assume without loss of generality that $A$ is free abelian and hence
torsion free.
On applying
the proof of Theorem \ref{abl} to the action of $A$ given by the restriction
of the action of $G$, we obtain the finite index subgroup $L$ of $A$
which has an action
preserving factors and by isometries
on the product $Q=\Gamma_1\times \ldots\times \Gamma_m$,
where we have that $\Gamma_i$ has the path metric $d_i$ and $Q$ has the
corresponding $\ell_2$ metric $d_Q$. Moreover our homomorphism
$\theta:L\rightarrow\Z^m$ is now injective because $A$ and hence $L$
is torsion free. We also have elements $a_1,\ldots ,a_m\in L$ where
$a_i$ acts loxodromically on $\Gamma_i$ and elliptically on $\Gamma_j$
for $j\neq i$. Now the subgroup
$L_0:=\langle a_1,\ldots ,a_m\rangle$ of $L$ has  $a_1,\ldots ,a_m$ as a
free basis and injectivity of $\theta$ means that $L_0$
has finite index in $L$ and so in $A$. We will now show that $L_0$
is undistorted in $G$.

We take basepoints $x_1\in\Gamma_1,\ldots ,x_m\in\Gamma_m$,
and set ${\bf x}=(x_1,\ldots ,x_m)$. As the basis element $a_i$ of $L_0$
is loxodromic when acting on $\Gamma_i$ and elliptic when acting on the
other factors, there are constants $M_{ij}>0$ for $1\leq i,j\leq m$
such that if $i\neq j$ then for all $n\in\Z$ we have
\[d_i(a_j^n(x_i),x_i)\leq M_{ij}\]
whereas for all points $y_i\in\Gamma_i$ and $n\in\Z$ we have
\[d_i(a_i^n(y_i),y_i)\geq M_{ii}|n|.\]

Thus in $Q$ we have for any element $a=a_1^{n_1}\ldots a_m^{n_m}$ of $L_0$ that
\[
d_Q(a({\bf x}),{\bf x})\geq
\frac{1}{\sqrt{m}}\sum_{i=1}^m
d_i(a_i^{n_i}a_1^{n_1}\ldots a_{i-1}^{n_{i-1}}a_{i+1}^{n_{i+1}}\ldots
a_m^{n_m}(x_i),x_i):=S_a.\]
But for any isometries $g_1,\ldots ,g_k,f$ of a metric space $(X,d)$
with $x\in X$, the triangle
inequality gives us
\[d(fg_1\ldots g_k(x),x)\geq d(f(x),x)-d(g_1(x),x)-\ldots -d(g_k(x),x).\]
Putting this into the above gives us that
\begin{eqnarray*}
S_a &\geq &\frac{1}{\sqrt{m}}\sum_{i=1}^m\left(
d_i(a_i^{n_i}(x_i),x_i)-\sum_{j\neq i}
d_i(a_j^{n_j}(x_i),x_i)\right)\\
& \geq &\frac{1}{\sqrt{m}}\sum_{i=1}^m\left(M_{ii}|n_i|
-\sum_{j\neq i}  M_{ij}   \right)\\
&\geq & \frac{1}{\sqrt{m}}
\left(K_1|n_1|-\epsilon_1+\ldots +K_m|n_m|-\epsilon_m\right)
\end{eqnarray*}
where $K_i:=M_{ii}>0$ and
\[\epsilon_i:=\sum_{j\neq i} M_{ij}.\]
This gives us that
\[S_a\geq \frac{1}{\sqrt{m}} \left(K(|n_1|+\ldots +|n_m|)-\epsilon\right)\]
where $K=\mbox{min}\{K_1,\ldots ,K_m\}>0$ and
$\epsilon=\epsilon_1+\ldots +\epsilon_m$. But $|n_1|+\ldots +|n_m|$ is the
word length of $a$ in the obvious word metric on $L_0$. This means that the
action of $L_0$ and hence $A$ on $Q$ is undistorted, thus $A$ is undistorted
in any finitely generated group $G$ with an action on $Q$ that restricts
to this action on $A$.
\end{proof}

We now give a few examples, based on the work of Conner in \cite{cn1},
\cite{cn2}, to show that we cannot have such straightforward results for
small subgroups in $\mathcal M$.\\
\hfill\\
Example 1: We have already seen that $\Z^\infty$ is in $\mathcal M$ and the
wreath product $\Z\,\wr\,\Z$, a metabelian group, can be thought of as the
semidirect product
\[\Z^\infty \rtimes\Z
=\langle t,a_i\,(i\in\Z)\,|ta_jt^{-1}=a_{j+1},[a_k,a_l]\,(j,k,l\in\Z)\rangle.\]
On taking an angle $\theta$ such that $e^{i\theta}$ is transcendental, we can
consider the Euclidean isometry $t(z)=ze^{i\theta}$ of $\C$, which is a
rotation about 0, along with the translation $a_0(z)=z+1$. Then $t$ and
$a_0$ generate a faithful copy of $\Z\,\wr\,\Z$ which has an isometric action
on $\R^2$ where the subgroup $\Z^\infty$ acts purely loxodromically. For the
other elements, which will all be rotations in this action, we can also
consider the isometric action of $\Z\,\wr\,\Z$ on $\R$ where this time
$t$ becomes the translation $t(x)=x+1$ and $\Z^\infty$ acts trivially.
Finally we can take the isometric action of $\Z\,\wr\,\Z$ on the product
$\R^2\times \R\times X$ of locally finite hyperbolic graphs
$\R^2\times\R\times X$ (where as before $X$ is the hyperbolic cone), which
shows that this group is in $\mathcal M$.\\
\hfill\\
Example 2: To obtain polycyclic groups which are not virtually abelian
but which are in $\mathcal M$, we note that \cite{cn1} Lemma 3.5 states that
if for $n\geq 2$ we have $\phi\in GL(n,\Z)$ is irreducible and has an
eigenvalue on the unit circle then $G=\Z^n\rtimes_\phi\Z$ has a faithful
isometric action on $\R^3$ by translations.\\
\hfill\\
Example 3: A specific example of the above given in \cite{cn2} Example 7.1
is $n=4$ with
$\Z^4=\langle a,b,c,d\rangle$ and where $\phi$ sends $a$ to $b$
to $c$ to $d$ to $a^{-1}b^2c^{-1}d^2$. The characteristic equation of the
automorphism $\phi$ is irreducible and
has two complex conjugate eigenvalues on the unit
circle, one real eigenvalue inside and one real eigenvalue $\lambda$
with $|\lambda|>1$. Thus this group is in $\mathcal M$. However the $\Z^4$
subgroup is distorted in $G$: for instance let $t\in G$ be the stable
letter, then we have
\[t^nat^{-n}=a^{\alpha_n}b^{\beta_n}c^{\gamma_n}d^{\delta_n}\] with
$\alpha_n,\ldots ,\delta_n$ determined by the recurrence relation
obtained from the characteristic equation. This means that
$|\alpha_n|,\ldots ,|\delta_n|$ grow exponentially like $\lambda^n$ but
the left hand side has word length $2n+1$.

We finish this section by examining the soluble groups in $\mathcal M_0$.
Whilst even finitely generated soluble groups in $\mathcal M_0$ need not be
virtually abelian, such as $C_p\wr \Z$, this example is not torsion free
or even virtually torsion free. We now show that this is the only
obstruction, obtaining a result which can be compared with Conner's
Theorem 3.2 in \cite{cn2}. This states that if a soluble group of finite
virtual cohomological dimension acts purely loxodromically by isometries
on some metric space then $G$ is virtually metabelian. 

\begin{thm} \label{solm}
A (virtually) soluble group $G$ in $\mathcal M_0$ which is
also (virtually) torsion free is finitely generated and virtually abelian.
\end{thm}
\begin{proof} 
We can drop to a finite index subgroup $H$
of $G$, so also in $\mathcal M_0$,
which is soluble, torsion free and preserves factors in its
action on $P=\Gamma_1\times\ldots\times\Gamma_n$. Now any finite index
subgroup of $H$ will also have these properties, so we can assume by
dropping down if necessary that the minimum length $d$ of a derived series
over all finite index subgroups of $H$ is equal to the derived length
of $H$; that is the smallest value of $d$ such that the $d$th derived
subgroup $H^{(d)}$ is trivial. If $d=1$ then we are done by Theorem
\ref{abl}. Otherwise if $d\geq 2$ we set $N$ equal to $H^{(d-1)}$, which is
an abelian normal subgroup of $H$. Although we do not
know that $H$ is finitely generated, Theorem \ref{abl} tells us that $N$ is
finitely generated and free abelian. Moreover, as in the proof of
Theorem \ref{abl},
we have $L$ of finite index $l$ say in $N$ with
free abelian basis $a_1,\ldots ,a_m$ (where $m\leq n$) and such that $a_i$ acts
loxodromically on (the possibly renumbered) $\Gamma_i$ but elliptically on
$\Gamma_j$ for $i\neq j$. By dropping down to a finite index subgroup of $H$
(here also called $H$),
we can assume that $L\unlhd H$. This is because $H$ acts by conjugation on
the subgroups of $N$ and for any $h\in H$ the subgroup $hLh^{-1}$ is an
index $l$ subgroup of $N$. But $N$ is finitely generated so has only finitely
many subgroups of index $l$, meaning that $L$ has a finite orbit under this
action of $H$ and thus the normaliser of $L$ has finite index in $H$.

For a given $h\in H$ with $L$ normal in this new $H$ (which still
contains $N$) and $1\leq i\leq m$, let
$ha_ih^{-1}=a_1^{n_{i1}}\ldots a_m^{n_{im}}$. As $h$ preserves factors, $a_i$
is loxodromic or elliptic in its action on $\Gamma_j$ if and only if
$ha_ih^{-1}$ is. But the latter type is determined by whether $n_{ij}$ is
non-zero (loxodromic) or zero (elliptic) from the inequalities in the
proof of Corollary \ref{unds}. Thus when $i\neq j$ we conclude that
$n_{ij}=0$. So we have $ha_ih^{-1}=a_i^k$ and $h^{-1}a_ih=a_i^l$ for
some $l,k$, which tells us that $h^{-1}ha_ih^{-1}h=a_i^{kl}$ and thus
$ha_ih^{-1}=a_i^{\pm 1}$ whenever $h\in H$. This means that we have a
subgroup $H_i$ of index (at most) 2 in $H$ and containing $a_1,\ldots ,a_m$
in which $a_i$ is a central element. So $L=\langle a_1,\ldots ,a_m\rangle$
is in the centre of $H_0:=H_1\cap \ldots \cap H_m$, which has finite index
in $H$.

Now $N$ is torsion free and contains $L$ with finite index, so is also
a copy of $\Z^m$ with $N\leq H_0$.
We now show that $N$ is also central in $H_0$. Given any $n\in N$ and
$h_0\in H_0$, we have $n^l\in L$ and $h_0n^lh_0^{-1}=n^l$.
Thus $h_0nh_0^{-1}$ is an element of $N$ whose $l$th power is $n^l$.
Now $N$ is a copy of $\Z^m$ and so the taking of $l$th roots within
$N$ is unique, giving $h_0nh_0^{-1}=n$.

We now apply Corollary \ref{impr}: for each
$1\leq i\leq m$ we have that $a_i$ is in the centre of $H_0$, so we
obtain (with group $H_0$, central element $a_i$ and graph $X=\Gamma_i$)
a homomorphism $\theta_i:H_0\rightarrow\Z$ and by putting these together
another homomorphism $\theta:H_0\rightarrow\Z^m$ with image having finite
index in $\Z^m$. As for the kernel $K$ of $\theta$, this consists of
elements which act elliptically on the renumbered graphs $\Gamma_1,
\ldots ,\Gamma_m$ (though they will act loxodromically on at least one
of the other graphs which were thrown away). In particular $L$ intersects
this kernel trivially and thus so does $N$ because it is torsion free and
$L$ has finite index in $N$.

Thus we have $\theta:H_0\rightarrow\Z^m$ with
$\theta(L)\leq\theta(N)\leq\theta(H_0)\leq\Z^m$. But as
$\theta(L)=\langle \theta(a_1),\ldots ,\theta(a_m)\rangle$
has finite index in $\Z^m$, this forces $\theta(N)$ to have finite index
in $\theta(H_0)$. So the pullback $H_1:=\theta^{-1}(\theta(N))$ has
finite index in $H_0$ and contains $N$ and $K$. Since $N\cap K$ is trivial,
this means that $\theta$ gives rise to a decomposition of $H_1$ as a semidirect
product $K\rtimes N$, which is actually a direct product $K\times N$
because $N$ is in the centre of $H_1$. But at the start we assumed that
$N$ was equal to the $(d-1)$th derived subgroup $H^{(d-1)}$ of $H$. This
means that the $(d-1)$th derived subgroup $H_1^{(d-1)}$ of $H_1$ lies
inside $N$. However the fact that $H_1=K\times N$ with $N$ abelian means
that all commutators of $H_1$, and hence all derived subgroups $H^{(i)}$
for $i\geq 1$, lie in $K$. As $d\geq 2$ we see that $H_1^{(d-1)}\leq K\cap N$
and so is trivial. Thus $H_1$ is a finite index subgroup of $H$ with derived
length less than $d$, which is a contradiction.
\end{proof}

\section{Abelianisation of centralisers}

We now turn to our third group theoretic property which covers
centralisers of infinite order elements. This will involve Theorem
\ref{main}, although here we can include groups in
$\mathcal M$ too.

To this end, let us review how this type of result proceeds in the
CAT(0) setting (for instance the proof of \cite{bh} II.6.12). Suppose
$G$ is a group acting isometrically on any CAT(0) space $X$ and $g$
is a central element of $G$ which is acting both loxodromically and
semisimply (the infimum of the displacement function is attained).
Then $g$ has an axis $A$, namely an isometrically embedded copy of
$\R$ invariant under $g$, on which $g$ acts as a non trivial
translation.
This axis need not be invariant under all
of $G$, but there is a product decomposition of $X=X_0\times A$
(which can be thought of as ``many parallel copies'' of $A$), with
each copy $\{x_0\}\times A$ invariant under $g$ and the product
decomposition invariant under all of $G$. We then have a form of
orthogonal projection, allowing us to define a homomorphism
$\theta$ from $G$ to $\R$ with $\theta(g)$ acting as it did on $A$.
In particular $\theta(g)$ has infinite order in the abelian group
$\R$, thus the image of $g$ in the abelianisation of $G$ has
infinite order.

In our case, using Theorem \ref{main}, we have:
\begin{co} \label{mapc}
Suppose that a group $G$ has a central element $g$ of infinite order. If
$G$ acts purely loxodromically
by isometries on a finite product of locally finite hyperbolic
graphs then there is a finite index subgroup $L$ of $G$ and $n>0$ such that
the element $g^n$ is in $L$ and has infinite order in the abelianisation
$L/L'$.
\end{co}
\begin{proof} First drop to a finite index subgroup $H$ of $G$ as in
Corollary \ref{ismprd} where the action of $H$
on $\Gamma_1\times\ldots\times\Gamma_k\times\R^m$ say
preserves factors and set $h:=g^n\in H$. By Proposition 4.4 (iii)
$h$ acts as a loxodromic element on one of these factors. If it is
$\Gamma_i$ then
by Theorem \ref{main} we obtain our finite index subgroup $L$ of
$H$ containing $g$ and our homomorphism $\theta$ onto $\Z$ in which
$\theta(g)$ has infinite order. Now as $\Z$ is abelian, it is the case that
$\theta$ will factor through the abelianisation map
$\alpha:L\rightarrow L/L'$, so that $\alpha(g)$ also has
infinite order in the abelianisation of $L$.

If however it is the $\R^m$ factor where $g$ acts loxodromically then
$g$ will also act semisimply so,
by the above discussion for CAT(0) groups where $X=\R^m$, the same
argument works except with a homomorphism to $\R$ rather than to $\Z$.
\end{proof}

This argument allows us to show that some important groups do not lie
in $\mathcal M$, let alone ${\mathcal M}_0$. 
\begin{co} \label{mcg} Let $MCG$ be the mapping
class group of either the closed orientable surface of genus $g\geq 3$,
or of genus at least 2 with at least one boundary component or at least
two punctures. For any
action of $MCG$ on a finite product of locally finite hyperbolic graphs,
no Dehn twist can act as a loxodromic element.
\end{co}
\begin{proof}
Given a Dehn twist $\gamma\in MCG$, we apply Corollary \ref{mapc} and
utilise the argument in \cite{bh} II.7.26. This shows that there is
a subgroup $G$ of $MCG$ in which $\gamma$ is central and where $G$
is isomorphic to the fundamental group of the unit tangent bundle
of $\Sigma_2$, where $\Sigma_g$ is the closed orientable surface of genus
$g$. Thus $G$ is
a central extension of $\Z=\langle\gamma\rangle$ by $\pi_1(\Sigma_2)$
and $\gamma$ is trivial in the abelianisation of $G$. This is the
fundamental group of a closed orientable 3-manifold with
$\widetilde{PSL(2,\R)}$ geometry, so any finite index subgroup $L$ of $G$
will be too. On taking any positive power $\gamma^n$ such that
$\gamma^n\in L$, we know that $L$ is a central extension of $\Z$
by $\pi_1(\Sigma_g)$ with this central $\Z$ being trivial in the
abelianisation of $L$. But $\gamma^n$ must be in the centre of $L$
so lies in this $\Z$, thus Corollary \ref{mapc} now applies.
\end{proof}

This result is certainly false if local finiteness is removed. Indeed
\cite{befu} shows that for any Dehn twist there exists an isometric
action of the mapping class group on a hyperbolic graph on which this
Dehn twist acts loxodromically.

We wish to extend this result to the automorphism and outer automorphism
groups of a free group. Here we do have the equivalent statement in \cite{rwad}
that a Dehn twist has finite order in the abelianisation of its centraliser
but we do not have an equivalent version once we start moving to finite index
subgroups. However we can avoid that problem completely by using the
following abstract group theoretic result. This came about by failing
to find an example illustrating the above where we have a finite index
subgroup
$H$ of $G$ and an element $h\in H$ which is central in $G$, such that $h$
has finite order in the abelianisation of $G$ but infinite order in the
abelianisation of $H$. Instead our result is that there are no examples.
First we put this statement into context by considering
the following cases:\\
\hfill\\
{\bf Example}: (i) Say we have a central element $h$ of $G$ and a finite
index subgroup $H$ in $G$, along with a homomorphism
$\theta:H\rightarrow\Z$ where $\theta(h)=1$, so certainly $h$ has infinite
order in the abelianisation of $H$. We might also have $k\in H$ with
$\theta(k)=0$ and an element $g\in G\setminus H$ with $gkg^{-1}=hk^{-1}$.
Then $\theta$ cannot be extended to $G$, even on rescaling, because it would
imply that
$\theta(h)=0$. However $G$ might possess a different homomorphism $\phi$ onto
$\Z$ where $\phi(k)=1$ and $\phi(h)=2$ (for instance if
$H=\Z^2=\langle h,k\rangle$ has index 2 in $G$ and $g$ has order 2).\\
(ii) Now suppose we are in the same setup as above but this time
$gkg^{-1}=hk$. Then $h$ cannot have infinite order in the abelianisation
of $G$ as it is a commutator in $G$. But we have $g^nkg^{-n}=h^nk$ and
if $H$ has finite index in $G$ then $g^n\in H$ for some $n>0$. Hence
$h^n$ is a commutator in $H$ too, thus $h$ did not have infinite order
in the abelianisation of $H$ to start with.

We can now present our general result.
\begin{thm} \label{cab}
Suppose that a group $G$ has a finite index subgroup $H$ and an infinite
order element $h\in H$ which is central in $G$. If there exists a
homomorphism $\theta:H\rightarrow\Z$ with $\theta(h)\neq 0$ then there
exists a homomorphism $\phi:G\rightarrow \Z$ with $\phi(h)\neq 0$.
\end{thm}
\begin{proof}
Let $G=g_1H\cup\ldots\cup g_lH$
and $K=Ker(\theta)$ which of course need not be normal in $G$. However
we can still consider the action by left multiplication of $G$ on the
left cosets of $K$ in $G$. We first note that each coset is uniquely
defined by $g_ih^jK=h^jg_iK$ for some $1\leq i\leq l$ and $j\in\Z$. Thus
on regarding these cosets as points we can think of them as forming $l$
disjoint copies of $\Z$, equipped with an action of $\langle h\rangle$
that preserves each of these copies setwise and shifts the elements up
by 1 in each copy. But $h$ being central in $G$ means that all elements
of $G$ preserve setwise this decomposition into copies of $\Z$, because if
$g(g_iK)=g_mh^nK$ for the appropriate values of $m$ and $n$ then
$g(g_ih^jK)=h^jg_mh^nK=g_mh^{n+j}K$. This also shows that we can describe
the action of $g$ as some permutation of these copies of $\Z$ along with
a translation on each copy.

Let us consider all possible bijections on our space of left cosets of $K$
in $G$ which act in this way. They form a group $P$
which can be thought of
as a copy of the semidirect product $\Z^l\rtimes Sym(l)$. Here an
element $(r_1,\ldots ,r_l)\in\Z^l$ acts by preserving setwise the copies
of $\Z$
and translating the first by $r_1$, the second by $r_2$ and so on. Note
that our central element $h$ corresponds to $(1,\ldots ,1)$. Meanwhile
$Sym(l)$ acts by permuting the copies of $\Z$ without any translating,
so that $\pi\in Sym(l)$ would send the point $g_ih^jK$ to
$g_{\pi(i)}h^jK$. The semidirect product structure can be described by
defining $e_i$ to be the translation by 1 on the $i$th copy of $\Z$ and
fixing pointwise all the other copies. As this generates the normal
subgroup $\Z^l$ of the semidirect product, we have
$\pi e_i\pi^{-1}=e_{\pi(i)}$ for $\pi\in Sym(l)$ and we can extend over
linear combinations of the $e_i$s.

First assume that $G$ is actually equal to $P$. As our description of the
semidirect product structure gives rise to a finite presentation of $P$,
we see that there is a surjective homomorphism $\phi$ say from $P$ to $\Z$
given by sending each $e_i$ to 1 (and obviously all $\pi$ to 0). Note that
$\phi(r_1,\ldots ,r_l)=r_1+\ldots +r_l$ and in particular we have
$\phi(h)=l\neq 0$ so we are done in this case. But in general our action
of $G$ provides us with a homomorphism from $G$ to a subgroup of $P$, so
we just restrict $\phi$ to that subgroup whereupon the image of $h$ still
maps to $l>0$.
\end{proof}

If $G$ is finitely generated then so is its abelianisation,
thus saying that there is a homomorphism $\phi$ from $G$ to $\Z$ with
$\phi(h)\neq 0$ is the same as saying that $h$ has infinite order in
the abelianisation of $G$. This need not be true if $G$ is not finitely
generated, so for completeness we provide a result that covers this case too.
\begin{co} \label{infa}
Suppose that a group $G$ has a finite index subgroup $H$ and an infinite
order element $h\in H$ which is central in $G$. If $h$ has infinite order
in the abelianisation of $H$ then $h$ has infinite order in the abelianisation
of $G$.
\end{co}
\begin{proof} If there is some counterexample then we reduce it to
the finitely generated case. We first note that for any subgroup of a group
$G$, if $s\in S$ is an element having finite order in the abelianisation of
$S$ then it does so in the abelianisation of $G$, because some positive
power $s^n$ will be a product of commutators in $S$ and so also in $G$.

Now suppose as in the claimed counterexample
that there is $n>0$ with $h^n$ a product
$\Pi\,[x_i,y_i]$ of commutators in $G$. Let $G_0=\langle h,x_i,y_i\rangle$
be the finitely generated subgroup of $G$ given by $h$ and the elements of
$G$ making up these commutators. Then $h$ also has finite order in the
abelianisation of $G_0$. As $H$ has finite index in $G$, we know
$S:=G_0\cap H$ has finite index in $G_0$ and also $h$ is in $S$. But we
assumed $h$ has infinite order in the abelianisation of $H$ and thus it
has infinite order in the abelianisation of $S$ by the first comment.
Also $h$ is of course central in $G_0$ because it was in $G$. Thus we can
apply Theorem \ref{cab} to $G_0$ and $S$ for a contradiction.
\end{proof}

We can now provide the promised improvement to Theorem \ref{main}.
\begin{co} \label{impr}
Suppose we have a group $H$ acting by automorphisms on the locally
finite hyperbolic graph $X$ and a loxodromic element $g\in H$ which is
central in $H$. Then no elements act parabolically and
there is a homomorphism $\phi$ from $H$ to
$\Z$ such that $\phi(g)\neq 0$ and with $Ker(\phi)$ consisting
precisely of the elliptic elements in $H$ under this action.
\end{co}
\begin{proof}
We can combine the approaches of Theorem \ref{main} (whose notation we use
here) and Theorem \ref{cab} by noting that the proof of Theorem \ref{main}
is unchanged if we restrict the pseudoaxis $P_g$ (which is $H$ invariant)
to the orbit under $H$ of $v_0$. This smaller set will be
a subcollection of those copies of $\Z$, including $O_0$. Moreover the proof
also goes through if we take $L$ to be the setwise stabiliser of $O_0$
instead of stabilising all copies of $\Z$, which we do so now, except that
$L$ need not be normal in $H$. However we still obtain a homomorphism
$\theta$ from $L$ to $\Z$ with $\theta(g)\neq 0$ and with kernel $K$.
Now the stabiliser of $v_0$ in $H$ is exactly $K$, so the proof of
Theorem \ref{cab} above can be visualised by replacing the left cosets
of $K$ in $H$ by the vertices in the orbit of $v_0$ under $H$. This gives
us our new homomorphism $\phi$ from $H$ to $\Z$ with $\phi(g)\neq 0$.

Now an element of $H$ will be in the kernel of $\phi$ and/or will be
elliptic/loxodromic if and only if the same is true for any proper power.
Thus we can assume that any element $h$ in $H$ is of the form
$(r_1,...,r_l)$ as in Theorem \ref{cab}. If $(r_1,\ldots ,r_l)$
is not the zero vector then this element acts with some point
having an unbounded orbit, thus
elements outside $Ker(\phi)$ are certainly loxodromic. Similarly if
any of $r_1,\ldots ,r_l$ is zero then this element fixes a vertex.
But as $H$ acts by isometries on $X$,
it is the case that $r_1=\ldots
=r_l$ anyway. To see this, if say $r_1\neq r_2$ then we have vertices
$v,w$ of $X$ such that $h^n(v)=g^{nr_1}(v)$ and
$h^n(w)=g^{nr_2}(w)$. But $2d(v,w)=d(h^n(v),h^n(w))+d(w,v)
=d(g^{nr_1}(v),g^{nr_2}(w))+d(g^{nr_2}(w),g^{nr_2}(v))
\geq d(g^{nr_1}(v),g^{nr_2}(v))$
so $d(g^{n(r_1-r_2)}(v),v)$ is bounded as $n$ varies. But $g$ is
not elliptic so we must have $r_1=r_2$.
\end{proof}
Note that the homomorphism $\theta:L\rightarrow\Z$ from Theorem \ref{main} is
actually the same as the restriction of $\phi$ in Corollary \ref{impr},
up to scaling. This is equivalent to asserting that
$Ker(\theta)=Ker(\phi)\cap L$, which is true
because $Ker(\phi)$ consists of all the elliptic elements and $Ker(\theta)$
was shown to consist of all the elliptic elements in $L$.

We can now apply these results to some important examples.
We have already seen the equivalent for the mapping class group in
Corollary \ref{mcg} and the proof below also works here by
considering centralisers of Dehn twists without having to deal
directly with finite index subgroups that may occur if our given action
does not preserve factors.
\begin{co} \label{aut}
Let $G$ be the automorphism group $Aut(F_n)$ or the outer automorphism
group $Out(F_n)$ of the free group $F_n$, where $n\geq 4$. Then in any
isometric action of $G$ on any finite product $P$ of locally finite
hyperbolic graphs, no Nielsen automorphism can act loxodromically.
\end{co}
\begin{proof}
Suppose $G$ has such an action on some $P$ and take any
Nielsen automorphism $\alpha\in G$. Then the centraliser $C:=C_G(\alpha)$
also has such an action, so we apply Corollary \ref{mapc} to $C$ to
obtain a finite index subgroup $D$ of $C$
and $k>0$ where $\alpha^k\in D$ and has infinite order in the abelianisation
of $D$. Now $\alpha^k$ is certainly also in the centre of $C$,
thus by applying Corollary \ref{infa} we have that $\alpha^k$ has infinite
order in the abelianisation of $C$ too and hence so does $\alpha$.
But it is shown in \cite{rwad} that for $n\geq 4$ the abelianisation
of $C_G(\alpha)$ is finite both in $Aut$ and in $Out$ (though for
$n=2$ or 3 in both $Out$ and $Aut$, it is shown that in fact
$\alpha$ does have infinite order in the abelianisation). 
\end{proof}

\section{Quasitrees and bounded valence graphs}

Our results on groups possessing some sort of proper action on finite
products of locally finite hyperbolic graphs have necessarily insisted
that infinite order elements act loxodromically, in order to exclude
actions such as that on the hyperbolic cone where every infinite order
element is parabolic. However there are some
important classes of hyperbolic spaces where every isometry is either
elliptic or loxodromic.

A quasitree is a geodesic metric space which is complete and quasi-isometric
to a simplicial tree equipped with the path metric. By \cite{man} Section 3,
any isometry of a quasitree is indeed either elliptic or loxodromic.
We say a quasitree is locally finite if it can be given the structure
of a locally finite graph such that its metric is equal to the path metric
of this graph. Therefore we have immediately from Corollary \ref{aut}:
\begin{co} \label{qu}
Let $G$ be either the mapping class group of a closed surface of genus 3
(or of genus at least 2 and with boundary and/or at least two punctures),
or $Aut(F_n)$ or $Out((F_n)$ for $n\geq 4$.

Then $G$ has no proper isometric action on a finite product of locally
finite quasitrees; indeed in any isometric action of any finite index
subgroup on any locally finite quasitree, we have that all powers of
Dehn twists/Nielsen automorphisms act as elliptic elements.
\end{co}

The surprise here for the case of the mapping class groups is
that \cite{bebf}, extending the results of \cite{befu},
showed that without the locally finite hypothesis there
is such an action, indeed it is shown that there is a QIE action as defined
in Section 3 where it was agreed that this was a very strong type of
proper action. In particular for any Dehn twist $\gamma$
there is a finite index subgroup $H$ of the
mapping class group which contains $\gamma$ and which
has an action on a quasitree with $\gamma$ acting loxodromically.

Finally, now suppose that we consider hyperbolic graphs which are not just
locally finite but of bounded valence. It is also the case here that all
isometries are either elliptic or loxodromic. To see this, we could
recall another notion of proper for a group $G$ acting isometrically on
a metric space $X$: we say the action is {\bf uniformly metrically
proper} if for all $R>0$ there is some $N_R\in\N$ such that for any $x\in X$
the set $\{g\in G:d(g(x),x)\leq R\}$ has size at most $N_R$.
But if $g$ is an element
of infinite order acting freely on a metric graph $\Gamma$ with valence
bounded by $d$ then the action of $\langle g\rangle$
must be uniformly metrically proper on $\Gamma$ because for any vertex
$v\in V(\Gamma)$ there are only finitely many vertices at distance at
most $R$ from $v$, with this number independent of the actual vertex
chosen. We can now quote \cite{osac} Lemma 3.4 (credited to Ivanov-Olshanskii)
which says that a uniformly metrically proper action of a group on a
hyperbolic space cannot be parabolic.

Therefore we obtain just as before:
\begin{co}
Let $G$ be either the mapping class group of a closed surface of genus 3
(or of genus at least 2 and with boundary and/or at least two punctures),
or $Aut(F_n)$ or $Out(F_n)$ for $n\geq 4$.

Then $G$ has no proper isometric action on a finite product of bounded valence
hyperbolic graphs, indeed in any isometric action of any finite index
subgroup on any bounded valence hyperbolic graph, we have that all powers
of Dehn twists/Nielsen automorphisms act as elliptic elements.
\end{co}

We finish with two related questions:\\
Question 1. For $n\geq 4$, is there an isometric action of $Aut(F_n)$
or $Out(F_n)$ on a finite product of quasitrees which is QIE?

By \cite{bebf} mentioned above, we have an affirmative answer for the mapping
class group but we can still ask:\\
Question 2. For $g\geq 3$, is there a finite index subgroup of the mapping
class group of a closed orientable surface of genus $g$ which acts by
isometries on a {\it proper} hyperbolic space $X$ where some (power of some)
Dehn twist acts loxodromically?

We have shown that the answer is no if in addition $X$ is a graph.


\begin{thebibliography}{99}

\bibitem{abhos} C.\,Abbott, S.\,H.\,Balasubramanya and D.\,Osin,
{\it Hyperbolic structures on groups},
Algebraic and Geometric Topology, {\bf 19} (2019) 1747--1835.

\bibitem{befu} M.\,Bestvina, K.\,Bromberg and K.\,Fujiwara,
{\it Constructing group actions on quasi-trees and applications to
mapping class groups},
Publcations math\'matiques de l'IH\'ES, {\bf 122} 2015.

\bibitem{bebf} M.\,Bestvina, K.\,Bromberg and K.\,Fujiwara,
{\it Proper actions on finite products of quasi-trees},
\texttt{http://arxiv.org/1905.10813}

\bibitem{bri} M.\,R.\,Bridson, 
{\it Semisimple actions of mapping class 
groups on CAT(0) spaces}, Geometry of Riemann surfaces, 1--14, London 
Math. Soc. Lecture Note Ser., 368, Cambridge Univ. Press, Cambridge, 2010.

\bibitem{bh} M.\,R.\,Bridson and A.\,Haefliger,
Metric spaces of non-positive curvature,
Springer-Verlag, Berlin, 1999.

\bibitem{meggd} J.\,O.\,Button,
{\it Aspects of non positive curvature for linear groups with no
infinite order unipotents},
Groups, Geometry and Dynamics {\bf 13} (2019) 277--292.

\bibitem{metr} J.\,O.\,Button,
{\it Groups acting faithfully on trees and properly on
products of trees},
\texttt{http://arxiv.org/1910.04614}


\bibitem{cn1} G.\,R.\,Conner,
{\it A class of finitely generated groups with irrational
translation numbers},
Arch. Math. {\bf 69} (1997) 265--274.

\bibitem{cn2} G.\,R.\,Conner,
{\it Discreteness properties of translation numbers in solvable groups},
J. Group Theory {\bf 3} (2000) 77--94.

\bibitem{foly} T.\,Foertsch and A.\,Lytchak,
{\it The de Rham decomposition theorem for metric spaces},
Geom. Funct. Anal. {\bf 18} (2008) 120--143.

\bibitem{grman} D.\,Groves and J.\,F.\,Manning,
{\it Dehn filling in relatively hyperbolic groups},
Israel J. Math {\bf 168} (2008) 317--429.

\bibitem{man} J.\,F.\,Manning,
{\it Quasi-actions on trees and property (QFA)},
J. London Math. Soc. {\bf 73} (2006) 84--108.

\bibitem{osac} D.\,Osin
{\it Acylindrically hyperbolic groups},
\texttt{http://arxiv.org/1304.1246}

\bibitem{rwad} M.\,Rodenhausen and R.\,D.\,Wade,
{\it Centralisers of Dehn twist automorphisms of free groups},
Math. Proc. Cambridge Philos. Soc. {\bf 159} (2015) 89--114.



\end{thebibliography}
\end{document}